\documentclass[10pt]{article}
\usepackage{latexsym, amsmath, amssymb, amsthm}
\usepackage{enumerate, mathrsfs, bbm}
\usepackage{graphicx, tikz}
\usepackage{authblk, setspace, cite}
\usepackage[top=1 in, bottom=1 in, left=1.2 in, right=1.2 in]{geometry}

\newtheorem{theorem}{Theorem}[section]
\newtheorem{lemma}[theorem]{Lemma}
\newtheorem{corollary}[theorem]{Corollary}

\newtheorem{remark}[theorem]{Remark}
\newtheorem{definition}[theorem]{Definition}
\newtheorem{example}[theorem]{Example}

\numberwithin{equation}{section}

\makeatletter
\def\blfootnote{\xdef\@thefnmark{}\@footnotetext}
\makeatother
 
\def\p{\partial}  \def\ora{\overrightarrow}
\def\ol{\overline}		\def\m{\mathbb}		
\def\O{\Omega}  \def\lam{\lambda}  \def\eps{\epsilon}  \def\a{\alpha}	\def\b{\beta}
	\def\wt{\widetilde}
\def\ls{\lesssim}
\def\be{\begin{equation}}     \def\ee{\end{equation}}

\title{Application of the Neumann heat kernel to prevent the finite time blowup}

\author{Xin Yang and Zhengfang Zhou}
\date{}

\begin{document}

\title{Prevention of blowup via Neumann heat kernel}
\author[]{Xin Yang and Zhengfang Zhou}
\date{}
\maketitle

\begin{abstract}
Consider the heat equation $u_t-\Delta u=0$ on a bounded $C^2$ domain $\O$ in $\m{R}^{n}(n\geq 2)$ with any positive initial data. If a superlinear radiation law $\frac{\p u}{\p n}=u^{q}$ with $q>1$ is imposed on a partial boundary $\Gamma_1\subseteq\p\O$ which has a positive surface area, then it has been known that the solution $u$ blows up in finite time. However, if the partial boundary, on which the superlinear radiation law is prescribed, is shrinking and is denoted as $\Gamma_{1,t}$ at time $t$, then the solution may exist globally  as long as the surface area $|\Gamma_{1,t}|$ of $\Gamma_{1,t}$ decays fast enough. This paper asks the question that how fast should $|\Gamma_{1,t}|$ decay in order to have a bounded global solution? This question is of significant importance in realistic situations, such as the temperature control within a certain safe range. By taking advantage of the Neumann heat kernel, we conclude that a polynomial decay $|\Gamma_{1,t}|\sim |\Gamma_1|(1+Ct)^{-\beta}$ with any $\beta>n-1$ suffices to ensure a bounded global solution. 
\end{abstract}

\bigskip
\bigskip

\blfootnote{2010 Mathematics Subject Classification. 35A01; 35B44; 35K20; 35K08; 35C15.}

\blfootnote{Key words and phrases. Global existence; Bounded solutions; Neumann heat kernel; Representation formula.}


\section{Introduction}
\subsection{Historical works}
Since the pioneering papers by Kaplan \cite{Kap63} and Fujita\cite{Fuj66}, the blow-up phenomenon of parabolic equations has been extensive studied in the literature for the Cauchy problems as well as the boundary value problems. We refer the readers to the surveys \cite{DL00, Lev90}, the books \cite{Hu11, QS07} and the references therein. One of the typical problems is the heat equation with Neumann boundary conditions in a bounded domain $\O$ of $\m{R}^{n}$:
\be\label{heat with Neumann}
\left\{\begin{array}{lll}
(\p_{t}-\Delta_{x})u(x,t)=0 &\text{in}& \Omega\times (0,T], \vspace{0.02in}\\
\dfrac{\partial u(x,t)}{\partial n(x)}=F\big(u(x,t)\big) &\text{on}& \p\O\times (0,T], \vspace{0.02in}\\
u(x,0)=\psi(x) &\text{in}& \Omega,
\end{array}\right.\ee
where $\frac{\p u}{\p n}$ denotes the exterior normal derivative and $F$ is a smooth function. This equation is used to model the heat conduction problem with a radiation law prescribed on the boundary of the material body. The local well-posedness of (\ref{heat with Neumann}) has been studied very well (see e.g. \cite{L-GMW91, Fri64, Lie96, LSU68}). Moreover, if $F$ is bounded on $\m{R}$, then the local solution can be extended globally (see e.g. \cite{L-GMW91}). However, if $F$ is unbounded, then the finite time blowup of the local solution may occur (see e.g. \cite{LP74, Wal75, RR97, L-GMW91, HY94}). Similarly, people also studied more general parabolic equations  and investigated whether there is finite time blowup. Once the finite time blowup happens, it is of great importance to estimate the lifespan (maximal existence time) of the solution. 
There has been developed various methods to deal with the upper bound of the lifespan (see \cite{Lev75} for a list of six methods). The lower bound estimate appeared much later but also draw much attention recently (see e.g. \cite{Kap63,PS06,PS07,PS09,Ena11, PPV-P10a, PPV-P10b, PP13, BS14, LL12, TV-P16, AD17, DS16, NY19}). The main approach in these works was the energy method. 

As a more general consideration for the realistic problems, the radiation may only occur on a small portion of the boundary. In other words, the nonlinear Neumann boundary condition $\frac{\p u}{\p n}=F(u)$ may be only imposed on a small partial boundary $\Gamma_{1}\subsetneqq \p\O$, while $\frac{\p u}{\p n}=0$ on the rest of the boundary $\Gamma_{2}$. Taking \cite{YZ16} as an example, it studied the disaster of the Space Shuttle Columbia, for which the heat radiation only occured on partial boundary $\Gamma_1$ (see Figure \ref{Fig, Columbia}) due to the damage there. 
\begin{figure}[!ht]
\centering
\begin{tikzpicture}[scale=0.6]
\begin{large}
\draw (-5,1/2)-- (0,1/2);
\draw (0,1/2)--(5/2,7/2);
\draw [domain=5/2:3] plot ({\x},{7/2+1/16-(\x-11/4)^2});
\draw (3,7/2)--(3,1/2);
\draw (3,1/2)--(4,1/2);
\draw (-5,-1/2)-- (0,-1/2);
\draw (0,-1/2)--(5/2,-7/2);
\draw [domain=5/2:3] plot ({\x},{-7/2-1/16+(\x-11/4)^2});
\draw (3,-7/2)--(3,-1/2);
\draw (3,-1/2)--(4,-1/2);
\draw [domain=90:270] plot ({cos(\x)-5},{1/2*sin(\x)});
\draw [dashed] [domain=0:360] plot({-5+1/8*cos(\x)},{1/2*sin(\x)});
\draw [dashed] [domain=0:360] plot({4+1/8*cos(\x)},{1/2*sin(\x)});
\path (-1.5,2.2) coordinate (A);
\draw (A) node [above] {$u$: temperature};
\draw [color=blue] [domain=110:150] plot({5/4+3/4*cos(\x)},{-2+3/4*sin(\x)});
\draw [color=blue] [domain=290:330] plot({5/4+3/4*cos(\x)},{-2+3/4*sin(\x)});
\path (5/4,-2+1/4) coordinate (B);
\draw (B) node [right] {$\Gamma_1$};
\path (5/4+1/4,2-1/4) coordinate (D);
\draw (D) node [below] {$\Omega$};
\path (-2,1/2) coordinate (F);
\draw (F) node [above] {$\Gamma_2$};
\path (1/2,-5/2) coordinate (E);
\end{large}
\end{tikzpicture}
\caption{Space Shuttle Columbia}
\label{Fig, Columbia}
\end{figure}
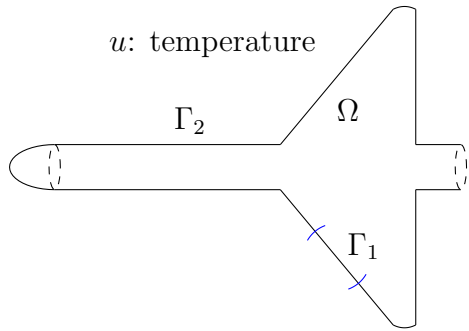
We refer the reader to that paper for the details of the background. Then it is natural to ask the following questions.
\begin{itemize}
\item [(1)] Will the finite time blowup still happen?
\item [(2)]If the finite time blowup occurs, then how does the lifespan depend on the size of $\Gamma_{1}$? 
\end{itemize}
In \cite{YZ16,YZ18,YZ19}, the authors investigated these questions for $F$ being a power function (see (\ref{Prob})) and they quantified both upper and lower bounds of the lifespan in terms of the surface area of $\Gamma_{1}$. Next, we will briefly summarize their main results.

Let $\Omega$ be a bounded open subset in $\m{R}^{n}$ ($n\geq 2$) with $C^{2}$ boundary $\partial\Omega$. \cite{YZ16} studied the following problem.
\be\label{Prob}
\left\{\begin{array}{rllll}
u_{t}(x,t)&=& \Delta u(x,t) &\text{in}& \Omega\times (0,T], \vspace{0.04in}\\
\dfrac{\partial u(x,t)}{\partial n(x)}&=& u^{q}(x,t) &\text{on}& \Gamma_1\times (0,T], \vspace{0.04in}\\
\dfrac{\partial u(x,t)}{\partial n(x)}&=& 0 &\text{on}& \Gamma_2\times (0,T], \vspace{0.04in}\\
u(x,0)&=& u_0(x) &\text{in}& \Omega,
\end{array}\right.\ee
where $\Gamma_1$ and $\Gamma_2$ are two disjoint open subsets of $\p\O$ with $\Gamma_1\neq\emptyset$, $\ol{\Gamma}_{1}\cup\ol{\Gamma}_{2}=\p\O$, and
\be\label{assumption on prob}
q>1,\, u_0\in C^{1}(\ol{\O}),\, u_0(x)\geq 0,\, u_0(x)\not\equiv 0. \ee
In addition, 
\[\wt{\Gamma}:= \ol{\Gamma}_1\cap\ol{\Gamma}_2\]
is a common $C^1$ boundary of $\Gamma_1$ and $\Gamma_2$. The normal derivative in (\ref{Prob}) is understood in the following way: for any $(x,t)\in\p\O\times(0,T]$,
\be\label{def of normal deri}
\frac{\p u(x,t)}{\p n(x)}:= \lim_{h\rightarrow 0^{+}} (Du)(x_h,t)\cdot\ora{n}(x),\ee
where $\ora{n}(x)$ denotes the exterior unit normal vector at $x$ and $x_h:= x-h\ora{n}(x)$ for $x\in\p\O$. According to \cite{YZ16}, a solution to (\ref{Prob}) up to time $T$ is defined as follows.
\begin{definition}\label{Def, soln to prob}(\cite{YZ16})
For $T>0$, a solution to (\ref{Prob}) on $\ol{\O}\times[0,T]$ means a function $u\in C^{2,1}(\O\times(0,T])\cap C(\overline{\O}\times[0,T])$ which has the following two properties.
\begin{enumerate}
\item [(1)] $u$ satisfies (\ref{Prob}) pointwise in the classical sense.
\item [(2)] For any $(x,t)\in \wt{\Gamma}\times (0,T]$, $\frac{\p u}{\p n}(x,t)$ exists and
\be\label{interface bdry deri}
\frac{\p u}{\p n}(x,t)=\frac{1}{2}\,u^{q}(x,t). \ee
\end{enumerate}
\end{definition}

\begin{definition}\label{Def, maximal soln}(\cite{YZ16})
The lifespan $T^{*}$ of (\ref{Prob}) is defined as
\[T^{*}:= \sup\big\{T\geq 0:\,\text{there exsits a solution to (\ref{Prob}) on}\,\,\, \ol{\O}\times[0,T]\big\}.\]
A function is called a maximal solution to (\ref{Prob}) if it solves (\ref{Prob}) up to the lifespan $T^{*}$.
\end{definition}

Based on these definitions, \cite{YZ16} concluded that $T^{*}$ is finite and positive as long as $|\Gamma_{1}|>0$, where $|\Gamma_1|$ denotes the surface area of $\Gamma_1$. In addition, there exists a unique nonnegative maximal solution to (\ref{Prob}). Moreover, if $\min\limits_{\ol{\O}}u_{0}>0$, then
\be\label{upper bdd}
T^{*}\leq \frac{1}{(q-1)|\Gamma_1|}\int_{\O}u_0^{1-q}(x)\,dx.\ee
Later in \cite{YZ18}, by denoting
\be\label{initial max}
M_0= \max_{x\in\ol{\O}}u_0(x),\ee
it provides a lower bound for $T^{*}$:
\be\label{lower bdd, old}
T^{*}\geq \frac{C}{q-1}\ln\Big(1+(2M_{0})^{-4(q-1)}\,|\Gamma_{1}|^{-\frac{2}{n-1}}\Big),\ee
where the constant $C$ only depends on $n$ and $\O$. It is worth mentioning that the lower bound (\ref{lower bdd, old}) does not require the convexity assumption on the domain $\O$, although this assumption was commonly seen in the previous works addressing the lower bound estimate. Combining (\ref{upper bdd}) and (\ref{lower bdd, old}) together, the asymptotic behavior of $T^{*}$ on $q$ is clear as $q\rightarrow 1^+$. That is: 
\[T^{*}\sim \frac{1}{q-1} \quad\text{as}\quad q\rightarrow 1^+.\]
But the asumptotic behavior of $T^{*}$ on $|\Gamma_1|$ as $|\Gamma_1|\rightarrow 0^+$ is far from satisfaction. In order to obtain more precise characterization, \cite{YZ19} took advantage of the Neumann heat kernel to achieve the following conclusion as $|\Gamma_{1}|\rightarrow 0^+$.
\begin{itemize}
\item If $n=2$, then 
\be\label{lifespan on Gamma_1, 2d}
|\Gamma_1|^{-1}\ln\big(|\Gamma_1|^{-1}\big)\ls T^{*}\ls |\Gamma_1|^{-1}.\ee
\item If $n\geq 3$, then
\be\label{lifespan on Gamma_1, hign dim}
|\Gamma_1|^{-\frac{1}{n-1}}\ls T^{*}\ls |\Gamma_1|^{-1}.\ee
\end{itemize}
In the two dimensional case, up to a logarithmic order, the order of $T^{*}$ is just $|\Gamma_1|^{-1}$. But in the higher dimensional cases, it is still an open question about the sharp order of $T^{*}$ on $|\Gamma_1|$ as $|\Gamma_1|\rightarrow 0^+$.

\subsection{Current problem and main results}
This paper focuses on a different aspect. In reality, the finite time blow-up is very dangerous, so rather than estimaing the blow-up time, it may be more desirable to take actions to prevent the blow-up. This paper will discuss one such possible way by repairing the heat radiation boundary. It turns out that as long as the surface area of the heat radiation boundary is decaying at some polynomial order, the temperature can be kept under a certain value (safe temperature). 

Let $\O$, $\Gamma_1$ and $\Gamma_2$ be the same as those in equation (\ref{Prob}). Let $\Gamma_{1,t}$ and $\Gamma_{2,t}$ be two boundary parts which are evolved from $\Gamma_{1}$ and $\Gamma_{2}$ at time $t$ such that the following three properties hold.
\begin{itemize}
\item[(A)] For any $t\geq 0$, $\Gamma_{1,t}$ and $\Gamma_{2,t}$ are two disjoint relatively open subsets of $\p\O$. Moreover, $\Gamma_{1,t}$ and $\Gamma_{2,t}$ share the common $C^1$ boundary $\wt{\Gamma}_t$, defined as in (\ref{common bdry}), such that $\p\O=\Gamma_{1,t}\cup\Gamma_{2,t}\cup\wt{\Gamma}_{t}$. 
\be\label{common bdry}
\p\Gamma_{1,t}=\p\Gamma_{2,t}:= \wt{\Gamma}_{t}.\ee

\item[(B)] There exists a continuous bijection $\Psi:\p\O\times[0,\infty)\rightarrow \p\O\times[0,\infty)$ such that for any $t\geq 0$, $\Psi(\Gamma_1\times\{t\})= \Gamma_{1,t}\times\{t\}$, $\Psi(\Gamma_2\times\{t\}) = \Gamma_{2,t}\times\{t\}$ and $\Psi(\wt{\Gamma}\times\{t\})= \wt{\Gamma}_{t}\times\{t\}$.

\item[(C)] $\Gamma_{1,t}$ is shrinking as $t$ is increasing, namely  $\Gamma_{1,t_{1}}\subseteq\Gamma_{1,t_{2}}$ if $t_{1}\geq t_{2}$.
\end{itemize}

\begin{example}
Let $\O$ be the unit ball in $\m{R}^{n}(n\geq 2)$: $\O=\{x\in\m{R}^{n}: |x|<1\}$.
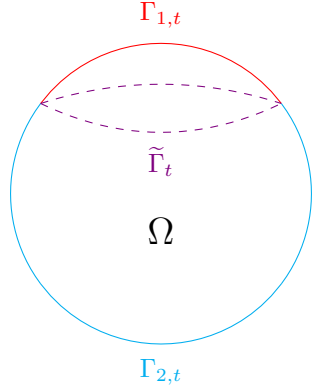
\begin{figure}[ht!]
\begin{center}
\begin{tikzpicture}[node distance=4]
\draw[domain=-1.6:1.6][red] plot (\x, {sqrt(4-(\x)^2)});
\node at (0,2.08) [above][red] {$\Gamma_{1,t}$};

\draw[domain=-1.2:1.2][cyan] plot ({sqrt(4-(\x)^2)}, \x);
\draw[domain=-1.2:1.2][cyan] plot ({-sqrt(4-(\x)^2)}, \x);
\draw[domain=-1.6:1.6][cyan] plot (\x, {-sqrt(4-(\x)^2)});
\node at (0,-2.08) [below][cyan] {$\Gamma_{2,t}$};

\draw[domain=-1.6:1.6][dashed][violet] plot (\x, {1.2+0.15*((\x)^2-2.56)});
\draw[domain=-1.6:1.6][dashed][violet] plot (\x, {1.2-0.1*((\x)^2-2.56)});
\node at (0,0.75) [below][violet] {$\wt{\Gamma}_{t}$};

\node at (0,-0.2) [below] {\Large $\Omega$};
\end{tikzpicture}
\caption{An Example of $\Gamma_{1,t}$ and $\Gamma_{2,t}$}
\label{Fig, region}
\end{center}
\end{figure}
For any point $x$ in $\m{R}^n$, we write it as $x=(\tilde{x},x_n)$, where $\tilde{x}\in\m{R}^{n-1}$ and $x_n\in\m{R}$. Define (see Figure \ref{Fig, region})
\[\Gamma_{1,t} = \big\{(\tilde{x},x_n): |\tilde{x}|<e^{-t},\, x_n=\sqrt{1-|\tilde{x}|^2}\big\}\]
and
\[\Gamma_{2,t} =\Big\{(\tilde{x},x_n):e^{-t}<|\tilde{x}|<1,\, x_n=\sqrt{1-|\tilde{x}|^2}\Big\} \bigcup \Big\{(\tilde{x},x_n): |\tilde{x}|\leq 1,\, x_n=-\sqrt{1-|\tilde{x}|^2}\Big\}.\]
Then 
\[\p\Gamma_{1,t}=\p\Gamma_{2,t}=\big\{(\tilde{x},x_n): |\tilde{x}|=e^{-t},\, x_n=\sqrt{1-|\tilde{x}|^2}\big\}:=\wt{\Gamma}_{t}.\]
It is clear that all the assumptions $(A)$, $(B)$ and $(C)$ are satisfied.

\end{example}

Based on the above notations, for any $T>0$, we decompose the lateral boundary $\p\O\times(0,T]$ into three parts: $S_{0,T}$, $S_{1,T}$ and $S_{2,T}$.
\be\label{bdry}\begin{split}
S_{1,T} &=\{(x,t): t\in(0,T],\,x\in\Gamma_{1,t}\},\\
S_{2,T} &=\{(x,t): t\in(0,T],\,x\in\Gamma_{2,t}\},\\
S_{0,T} &=\{(x,t): t\in(0,T],\,x\in\wt{\Gamma}_{t}\}.
\end{split}\ee
Then the problem (\ref{Prob, vary bdry}) below will be studied. 
\be\label{Prob, vary bdry}
\left\{\begin{array}{lrlll}
u_{t}(x,t) &=& \Delta u(x,t) &\text{in}& \Omega\times (0,T],\vspace{0.04in}\\
\dfrac{\partial u(x,t)}{\partial n(x)} &=& u^{q}(x,t) &\text{on}& S_{1,T},\vspace{0.04in}\\
\dfrac{\partial u(x,t)}{\partial n(x)} &=& 0 &\text{on}& S_{2,T},\vspace{0.04in}\\
u(x,0) &=& u_0(x) &\text{in}& \Omega.
\end{array}\right.\ee
We use $|\Gamma_{1}|$ and $|\Gamma_{1,t}|$ to denote the surface areas of $\Gamma_{1}$ and $\Gamma_{1,t}$ respectively, that is, 
\[|\Gamma_{1}|=\int_{\Gamma_{1}}\,dS(x)\quad\text{and}\quad |\Gamma_{1,t}|=\int_{\Gamma_{1,t}}\,dS(x),\]
where $dS(x)$ means the surface integral with respect to the variable $x$. For convenience of notation, we define
\be\label{area fn}
A(t)=|\Gamma_{1,t}|.\ee
In this paper, it is assumed that $A(t)=|\Gamma_1|f(t)$, where $f$ is a decreasing function from  $[0,\infty)$ to $(0,1]$ with $f(0)=1$. Then we are asking how fast $f(t)$ should decrease to prevent $u$ blowing up in finite time or to prevent $u$ exceeding a certain value (say a safe temperature). Should $f(t)$ decay exponentially like $f(t)\sim |\Gamma_1|e^{-Ct}$? Or should a polynomial decay like $f(t)\sim (1+Ct)^{-\beta}$ be enough? If a polynomial decay suffices, then how large the decay power $\beta$ is needed? We will answer these questions in Theorem \ref{Thm, prevent finite time blow-up} and Theorem \ref{Thm, control temp}. Roughly speaking, a polynomial decay with power $\beta>n-1$ fulfills our expectations.

Similar to Definition \ref{Def, soln to prob} and Definition \ref{Def, maximal soln} for the problem (\ref{Prob}), the local solution to (\ref{Prob, vary bdry}) and its lifespan are understood in the sense of Definition \ref{Def, soln to prob with vary bdry} and Definition \ref{Def, maximal soln with vary bdry}. 
\begin{definition}\label{Def, soln to prob with vary bdry}
For any $T>0$, a solution to (\ref{Prob, vary bdry}) on $\ol{\O}\times[0,T]$ means a function $u\in C^{2,1}(\O\times(0,T])\cap C(\overline{\O}\times[0,T])$ which has the following two properties.
\begin{enumerate}
\item [(1)] $u$ satisfies (\ref{Prob, vary bdry}) pointwise in the classical sense.
\item [(2)] For any $(x,t)\in S_{0,T}$, $\frac{\p u}{\p n}(x,t)$ exists and
\be\label{interface bdry deri. for simple model}
\frac{\p u}{\p n}(x,t)=\frac{1}{2}\,u^{q}(x,t). \ee
\end{enumerate}
\end{definition}

\begin{definition}\label{Def, maximal soln with vary bdry}
The lifespan $T^{*}$ of (\ref{Prob, vary bdry}) is defined as
\[T^{*}:= \sup\big\{T\geq 0:\,\text{there exsits a solution to (\ref{Prob, vary bdry}) on}\,\,\, \ol{\O}\times[0,T]\big\}.\]
A function is called a maximal solution to (\ref{Prob, vary bdry}) if it solves (\ref{Prob, vary bdry}) up to the lifespan $T^{*}$.
\end{definition}

Throughout this paper, we define $M_0$ as in (\ref{initial max}) to be the maximum of the initial data. In addition, we denote $M(t)$ to be the supremum of the solution $u$ to (\ref{Prob, vary bdry}) on $\ol{\O}\times[0,t]$:
\be\label{max function at time t}
M(t)=\sup_{(x,\tau)\in \ol{\O}\times[0,t]}u(x,\tau).\ee
Based on the proofs of Theorem 1.3 and Corollary 1.1 in \cite{L-GMW91}, and Appendix B in \cite{YZ16}, we are able to derive the fundamental result on the existence and uniqueness of the solution to (\ref{Prob, vary bdry}).

\begin{theorem}\label{Thm, fund thm for vary bdry}
Assume (\ref{assumption on prob}), (A), (B) and (C) hold. Then the lifespan $T^{*}$ of (\ref{Prob, vary bdry}) is positive (possibly infinity) and there exists a unique maximal solution $u\in C^{2,1}\big(\O\times(0,T^{*})\big)\cap C\big(\ol{\O}\times[0,T^{*})\big)$ to (\ref{Prob, vary bdry}). Moreover, $u(x,t)>0$ for any $(x,t)\in\ol{\O}\times(0,T^{*})$. Finally, if $T^{*}<\infty$, then 
\be\label{bdd soln can extend}
\lim\limits_{t\nearrow T^{*}}M(t)=\infty,\ee
where $M(t)$ is defined as in (\ref{max function at time t}).
\end{theorem}

Theorem \ref{Thm, fund thm for vary bdry} states that the lifespan of (\ref{Prob, vary bdry}) is just the blow-up time of the supremum norm of its solution $u$. So we will also call $T^{*}$ to be the blow-up time. Moreover, this theorem indicates that in order to obtain a global solution, one just needs to ensure the solution to be bounded at any finite time $T$. Based on this observation, we will show that the finite time blow-up can not happen as long as the surface area decays as (\ref{surface area decay, global soln}). As a convention of the notations, $C=C(a,b\dots)$ and $C_{i}=C_{i}(a,b\dots)$ in this paper will stand for positive constants which only depend on the parameters $a,b\dots$. In addition, $C$ and $C_i$ may represent different constants from line to line.

\begin{theorem}\label{Thm, prevent finite time blow-up}
Assume (\ref{assumption on prob}), (A), (B) and (C) hold. Let $T^{*}$ be the lifespan of (\ref{Prob, vary bdry}). Define $M_{0}$ as in (\ref{initial max}). Then for any $\beta>n-1$, there exists $C^{*}=C^{*}(n,\O, q, \beta, M_{0}, |\Gamma_{1}|)$ such that if 
\be\label{surface area decay, global soln}
A(t)\leq |\Gamma_{1}|(1+C^{*}t)^{-\beta},\ee
then $T^{*}=\infty$. 
\end{theorem}

In many realistic situations, it becomes very dangerous once the temperature reaches a high value. So it is of great importance for not only preventing the finite time blow-up, but also keeping the temperature below a safe limit. The next theorem accomplishes this task.

\begin{theorem}\label{Thm, control temp}
Assume (\ref{assumption on prob}), (A), (B) and (C) hold. Let $T^{*}$ be the lifespan of (\ref{Prob, vary bdry}) whose maximal solution is denoted as $u$. Define $M_{0}$ as in (\ref{initial max}). Then for any $B>M_{0}$ and $\beta>n-1$, there exists a constant $C_B^*=C_B^*(n,\O,q,\b, M_{0}, |\Gamma_{1}|, B)$ such that if 
\be\label{surface area decay, bdd soln}
A(t)\leq |\Gamma_{1}|(1+C_{B}^{*}t)^{-\beta},\ee
then $T^{*}=\infty$ and $u(x,t)\leq B$ for any $x\in\ol{\O}$ and $t\geq 0$.
\end{theorem}

\subsection{Organization of the paper}
The paper is organized as follows. Section \ref{Sec, pre} will introduce some preliminary results, including the properties of the Neumann heat kernel in Corollary \ref{Cor, prop of NHK}, the representation formula (\ref{rep for soln}) and the estimate for the boundary-time integral of the Neumann heat kernel in Lemma \ref{Lemma, bound by power of Gamma, NHK}. Then in Section \ref{Sec, proof of global soln} and Section \ref{Sec, proof of bdd soln}, the main results of this paper, Theorem \ref{Thm, prevent finite time blow-up} and \ref{Thm, control temp}, will be proved 
respectively. Finally, Section \ref{Sec, fun lemma} presents an elementary lemma which is used in Remark \ref{Re, sharpness of log growth} and \ref{Re, choice of beta} to illustrate the necessity (by the method in this paper) of the condition $\beta>n-1$ in Theorem \ref{Thm, prevent finite time blow-up} and \ref{Thm, control temp}.

\section{Preliminaries}\label{Sec, pre}
Throughout this paper,  $\Phi$ refers to the heat kernel of $\m{R}^{n}$:
\be\label{fund soln of heat eq}
\Phi(x,t)=\frac{1}{(4\pi t)^{n/2}}\,\exp\Big(-\frac{|x|^2}{4t}\Big), \quad\forall\, (x,t)\in\m{R}^{n}\times(0,\infty).\ee

\subsection{Neumann Green's function and Neumann heat kernel}
\label{Subsec, NGF and NHK}
Given a bounded domain $\O$ in $\m{R}^{n}$, one can define the fundamental solution associated to the heat operator 
\be\label{heat op}
L_{tx}=\p_{t}-\Delta_{x}\ee
on $\O$ (see e.g. \cite{Dre40, Fel37, Ito53}). If in addition the boundary conditions are considered, then one can also study the fundamental solution adapted to the boundary conditions (see e.g. \cite{Ito54, Ito57a, Ito57b}). In particular, if the boundary condition is of Neumann type, then the associated fundamental solution is called the Neumann Green's function. We follow (\cite{Ito54}, Page 171) to define the Neumann Green's function in Definition \ref{Def, NGF for heat}. As a convenience of notation, we denote 
\be\label{compatible initial}
C_{N0}(\ol{\O})=\Big\{\psi\in C(\ol{\O}):\frac{\p\psi}{\p n}=0 \,\,\text{on}\,\, \p\O\Big\}.\ee
\begin{definition}[\cite{Ito54}, \cite{YZ19}]\label{Def, NGF for heat}
Let $\O$ be a bounded domain in $\m{R}^{n}$ with $C^{2}$ boundary $\p\O$. Define the Neumann Green's function for the heat operator $\p_{t}-\Delta_{x}$ in $\O$ to be a continuous function $G(x,t,y,s)$ on $\{(x,t,y,s): x,y\in\ol{\O}, t,s\in\m{R}, s<t\}$ such that for any $s\in\m{R}$ and for any $\psi\in C_{N0}(\ol{\O})$, the function $v(x,t)$ defined as 
\be\label{def of v}
v(x,t)=\int_{\O}G(x,t,y,s)\psi(y)\,dy, \quad\forall\,x\in\ol{\O},\, t>s,\ee
belongs to $C^{2,1}\big(\ol{\O}\times (s,\infty)\big)$ and satisfies (\ref{simple soln by Green fn}).
\be\left\{\label{simple soln by Green fn}\begin{array}{rlll}
(\p_{t}-\Delta_{x})v(x,t) &=& 0, & \quad\forall\,x\in\ol{\O},\,t>s, \vspace{0.02in}\\
\dfrac{\p v(x,t)}{\p n(x)} &=& 0, & \quad\forall\,x\in\p\O,\, t>s, \vspace{0.02in}\\
\lim\limits_{t\rightarrow s^{+}} v(x,t) &=& \psi(x), & \quad\text{uniformly in $x\in\ol{\O}$}. 
\end{array}\right.\ee
\end{definition}
The following lemma demonstrates the existence and uniqueness of the Neumann Green's function as well as some of its basic properties. 

\begin{lemma}\label{Lemma, prop of Green fn}
Let $\O$ be a bounded domain in $\m{R}^{n}$ with $C^{2}$ boundary $\p\O$. Then there exists a unique Neumann Green's function $G(x,t,y,s)$ for the heat operator $\p_{t}-\Delta_{x}$ in $\O$. Moreover, it has the following properties.
\begin{enumerate}[(a)]
\item $G(x,t,y,s)$ is $C^{2}$ in $x$ and $y$ ($x,y\in\ol{\O}$), and $C^{1}$ in $t$ and $s$ ($s<t$).

\item For fixed $s\in\m{R}$ and $y\in\ol{\O}$, as a function in $x$ and $t$ ($x\in\ol{\O}$ and $t>s$), $G(x,t,y,s)$ satisfies (\ref{Green fn solves heat and bdry}).
\be\left\{\label{Green fn solves heat and bdry}\begin{array}{rl}
(\p_{t}-\Delta_{x})G(x,t,y,s) = 0, &\quad\forall\,x\in\ol{\O},\, t>s, \vspace{0.02in}\\
\dfrac{\p G(x,t,y,s)}{\p n(x)} = 0, & \quad\forall\,x\in\p\O,\, t>s.
\end{array}\right.\ee

\item For any $s\in\m{R}$ and $\psi\in C_{N0}(\ol{\O})$, the function $v(x,t)$ defined in (\ref{def of v}) is the unique function in $C^{2,1}\big(\ol{\O}\times (s,\infty)\big)$ that satisfies (\ref{simple soln by Green fn}).

\item $G(x,t,y,s)\geq 0$ for any $x,y\in\ol{\O}$ and $s<t$.
\item $\int_{\O}G(x,t,y,s)\,dy=1$ for any $x\in\ol{\O}$ and $s<t$.
\item For any $x,y\in\ol{\O}$ and $s<t$, 
\[G(x,t,y,s)=G(x,t-s,y,0) \quad\text{and}\quad G(x,t,y,s)=G(y,t,x,s).\]
\end{enumerate}
\end{lemma}
\begin{proof}
See Lemma 2.2 in \cite{YZ19}.
\end{proof}

From the above property (f), the Neumann Green's function is invariant under the time translation. Define \be\label{def of NHK}
N(x,y,t)=G(x,t,y,0).\ee
This function is called the Neumann heat kernel of $\O$ and it has the property that $G(x,t,y,s)=N(x,y,t-s)$. Based on this observation, an equivalent definition of the Neumann heat kernel is given as below.

\begin{definition}[\cite{YZ19}]\label{Def, NHK}
Let $\O$ be a bounded domain in $\m{R}^{n}$ with $C^{2}$ boundary $\p\O$. A function $N(x,y,t)$ on $\ol{\O}\times\ol{\O}\times(0,\infty)$ is called a Neumann heat kernel if 
the function $G(x,t,y,s)$ defined by 
\[G(x,t,y,s)=N(x,y,t-s)\]
is the Neumann Green's function as defined in Definition \ref{Def, NGF for heat}.
\end{definition}

Combining (\ref{def of NHK}) and Lemma \ref{Lemma, prop of Green fn}, we list some properties of the Neumann heat kernel.

\begin{corollary}\label{Cor, prop of NHK}
Let $\O$ be a bounded domain in $\m{R}^{n}$ with $C^{2}$ boundary $\p\O$. Then there exists a unique Neumann heat kernel $N(x,y,t)$ of $\O$ as in Definition \ref{Def, NHK}. In addition, it has the following properties. 
\begin{enumerate}[(a)]
\item $N(x,y,t)$ is $C^{2}$ in $x$ and $y$ ($x,y\in\ol{\O}$), and $C^{1}$ in $t$ ($t>0$).

\item For fixed $y\in\ol{\O}$, as a function in $(x,t)$, $N(x,y,t)$ satisfies (\ref{NHK solves heat and bdry}).
\be\left\{\label{NHK solves heat and bdry}\begin{array}{rl}
(\p_{t}-\Delta_{x})N(x,y,t)=0, & \quad\forall\,x\in\ol{\O},\, t>0, \vspace{0.02in}\\
\dfrac{\p N(x,y,t)}{\p n(x)}=0, & \quad\forall\,x\in\p\O,\, t>0.
\end{array}\right.\ee

\item For any $\psi\in C_{N0}(\ol{\O})$ , the function $w(x,t)$ defined by
\be\label{def of w}
w(x,t)=\int_{\O}N(x,y,t)\psi(y)\,dy\ee
is the unique function in $C^{2,1}\big(\ol{\O}\times (0,\infty)\big)$ that satisfies (\ref{simple soln by NHK}).
\be\left\{\label{simple soln by NHK}\begin{array}{rlll}
(\p_{t}-\Delta_{x})w(x,t) &=& 0,  &\quad\forall\,x\in\ol{\O},\,t>0, \vspace{0.02in}\\
\dfrac{\p w(x,t)}{\p n(x)} &=& 0, &\quad\forall\,x\in\p\O,\, t>0, \vspace{0.02in}\\
\lim\limits_{t\rightarrow 0^{+}} w(x,t) &=& \psi(x), &\quad\text{uniformly in $x\in\ol{\O}$}. 
\end{array}\right.\ee

\item $N(x,y,t)\geq 0$ and $N(x,y,t)=N(y,x,t)$  for any $x,y\in\ol{\O}$ and $t>0$.
\item For any $x\in\ol{\O}$ and $t>0$,
\be\label{NHK int id}\int_{\O}N(x,y,t)\,dy=1.\ee
\end{enumerate}
\end{corollary}
\begin{proof}
These are direct consequences of Definition \ref{Def, NGF for heat}, Lemma \ref{Lemma, prop of Green fn} and Definition \ref{Def, NHK}.
\end{proof}

Unlike the heat kernel $\Phi$ of $\m{R}^{n}$ in (\ref{fund soln of heat eq}), $N(x,y,t)$ in general does not have an explicit formula. Nevertheless, when $t$ is small, $N(x,y,t)$ can be dominated in terms of $\Phi$.

\begin{lemma}\label{Lemma, quant of NHK}
There exists $C=C(n,\O)$ such that for any $x,y\in\ol{\O}$ and $t\in(0,1]$,
\be\label{quant of NHK}
0\leq N(x,y,t)\leq C\,\Phi(x-y,2t).\ee
\end{lemma}
\begin{proof}
See Lemma 2.5 in \cite{YZ19}. \end{proof}

\subsection{Representation formula by the Neumann heat kernel}
\label{Subsec, Rep formula by NHK}
One of the applications of the Neumann heat kernel is the representation formula of the solution to (\ref{Prob, vary bdry}). As a heuristic argument, let's fix any $x\in\O$ and $t>0$ and pretend the solution $u$ to (\ref{Prob, vary bdry}) is sufficiently smooth up to the boundary. Then it follows from part (b) and (d) of Corollary \ref{Cor, prop of NHK} that 
\[(\p_{t}-\Delta_{y})N(x,y,t-\tau)=(\p_{t}-\Delta_{y})N(y,x,t-\tau)=0, \quad\forall\, y\in\ol{\O},\,0<\tau<t.\]
As a result, 
\[\int_{0}^{t}\int_{\O}(\p_{t}-\Delta_{y})N(x,y,t-\tau)\,u(y,\tau)\,dy\,d\tau=0.\]
Equivalently,
\[\int_{0}^{t}\int_{\O}(-\p_{\tau}-\Delta_{y})N(x,y,t-\tau)\,u(y,\tau)\,dy\,d\tau=0.\]
Now formally integrating by parts and taking advantage of (b), (c) and (d) in Corollary \ref{Cor, prop of NHK}, we obtain
\[\begin{split}
u(x,t)=& \int_{0}^{t}\int_{\O}N(x,y,t-\tau)\,(\p_{\tau}-\Delta_{y})u(y,\tau)\,dy\,d\tau+\int_{\O}N(x,y,t)\,u(y,0)\,dy\\
&+\int_{0}^{t}\int_{\p\O}N(x,y,t-\tau)\,\frac{\p u(y,\tau)}{\p n(y)}\,dS(y)\,d\tau. 
\end{split}\]
Keeping in mind that $u$ is the solution to (\ref{Prob, vary bdry}), so
\[u(x,t)=\int_{\O}N(x,y,t)u_{0}(y)\,dy+\int_{0}^{t}\int_{\Gamma_{1,\tau}}N(x,y,t-\tau)u^{q}(y,\tau)\,dS(y)\,d\tau.\]
We formally state the above result in Lemma \ref{Lemma, rep for soln, initial}. The rigorous proof can be carried out by similar argument as in Appendix A of \cite{YZ19}.
\begin{lemma}\label{Lemma, rep for soln, initial}
Let $u$ be the maximal solution to (\ref{Prob, vary bdry}) with the lifespan $T^{*}$. Then for any $(x,t)\in\ol{\O}\times(0,T^{*})$, 
\be\label{rep for soln, initial}
u(x,t)=\int_{\O}N(x,y,t)u_{0}(y)\,dy+\int_{0}^{t}\int_{\Gamma_{1,\tau}}N(x,y,t-\tau)u^{q}(y,\tau)\,dS(y)\,d\tau.\ee
\end{lemma}

\begin{corollary}\label{Cor, rep for soln}
Let $u$ be the maximal solution to (\ref{Prob, vary bdry}) with the lifespan $T^{*}$. Then for any $T\in [0,T^{*})$ and for any $(x,t)\in\ol{\O}\times(0,T^{*}-T)$,
\be\label{rep for soln}
u(x,T+t)=\int_{\O}N(x,y,t)u(y,T)\,dy+\int_{0}^{t}\int_{\Gamma_{1, T+\tau}}N(x,y,t-\tau)u^{q}(y,T+\tau)\,dS(y)\,d\tau.\ee
\end{corollary}
\begin{proof}
Regarding $u(\cdot,T)$ as the new initial data and then applying Lemma \ref{Lemma, rep for soln, initial} leads to the conclusion. \end{proof}

\subsection{Boundary-time integral of the Neumann heat kernel}
\label{Subsec, bdry-time int est}

\begin{lemma}\label{Lemma, bound by power of Gamma, NHK}
Let $\alpha\in\big[0,\frac{1}{n-1}\big)$ and let $N(x,y,t)$ be the Neumann heat kernel as in Definition \ref{Def, NHK}. Then there exists $C=C(n,\O,\a)$ such that for any $\Gamma\subseteq\p\O$, $x\in\ol{\O}$ and $t\in[0,1]$,
\be\label{bound by power of gamma, NHK}
\int_{0}^{t}\int_{\Gamma}N(x,y,t-\tau)\,dS(y)\,d\tau\leq C \,|\Gamma|^{\alpha}\,t^{[1-(n-1)\alpha]/2}.\ee
\end{lemma}
\begin{proof}
Performing the change of variable $\tau\rightarrow t-\tau$ leads to 
\[\int_{0}^{t}\int_{\Gamma}N(x,y,t-\tau)\,dS(y)\,d\tau=\int_{0}^{t}\int_{\Gamma}N(x,y,\tau)\,dS(y)\,d\tau.\]
Taking advantage of Lemma \ref{Lemma, quant of NHK} and the fact $t\leq 1$, there exists a constant $C=C(n,\O)$ such that 
\begin{eqnarray*}
\int_{0}^{t}\int_{\Gamma}N(x,y,\tau)\,dS(y)\,d\tau &\leq & C\int_{0}^{t}\int_{\Gamma}\Phi(x-y,2\tau)\,dS(y)\,d\tau\\
&=& C\int_{0}^{2t}\int_{\Gamma}\Phi(x-y,\tau)\,dS(y)\,d\tau.
\end{eqnarray*}
Finally, applying Lemma 2.9 in \cite{YZ19}, there exists a constant $C=C(n,\O,\a)$ such that 
\begin{eqnarray*}
\int_{0}^{2t}\int_{\Gamma}\Phi(x-y,\tau)\,dS(y)\,d\tau &\leq & C |\Gamma|^{\a}(2t)^{[1-(n-1)\a]/2}\\
&=& C |\Gamma|^{\a}t^{[1-(n-1)\a]/2}.
\end{eqnarray*}

\end{proof}

\section{Proof of Theorem \ref{Thm, prevent finite time blow-up}}
\label{Sec, proof of global soln}
The following lemma is in the same spirit as Lemma 3.1 in \cite{YZ19} which characterizes how fast the supremum norm of the solution can grow. This is an essential estimate that will be used in the proofs of Theorem \ref{Thm, prevent finite time blow-up} and Theorem \ref{Thm, control temp}.

\begin{lemma}\label{Lemma, growth rate, general}
Let $u$ be the maximal solution to (\ref{Prob, vary bdry}) with the lifespan $T^{*}$. Define $M(t)$ and $A(t)$ as in (\ref{max function at time t}) and (\ref{area fn}). Then for any $\alpha\in[0,\frac{1}{n-1})$, there exists $C=C(n,\O,\alpha)$ such that for any $T\in [0,T^{*})$ and for any $0\leq t<\min\{1,T^{*}-T\}$,  
\be\label{growth rate, general}
\frac{M(T+t)-M(T)}{M^{q}(T+t)}\leq C\,[A(T)]^{\alpha}\,t^{[1-(n-1)\alpha]/2}.\ee
\end{lemma}
\begin{proof}
It is equivalent to prove 
\[M(T+t)\leq M(T)+CM^{q}(T+t)\,[A(T)]^{\alpha}\,t^{[1-(n-1)\alpha]/2}.\]
That is to show for any $x\in\ol{\O}$ and for any $\sigma\in[0,T+t]$,
\be\label{equiv for growth}
u(x,\sigma)\leq M(T)+CM^{q}(T+t)\,[A(T)]^{\alpha}\,t^{[1-(n-1)\alpha]/2}.\ee
Fix any $x\in\ol{\O}$ and $\sigma\in[0,T+t]$. There are two cases.
\begin{itemize}
\item Firstly, $\sigma\in[0,T]$. In this case, it is obvious that $u(x,\sigma)\leq M(T)$, which implies (\ref{equiv for growth}).

\item Secondly, $\sigma\in(T,T+t]$. By the representation formula (\ref{rep for soln}) with $t=\sigma-T$,
\begin{align*}
u(x,\sigma) &= \int_{\O}N(x,y,\sigma-T)u(y,T)\,dy+\int_{0}^{\sigma-T}\int_{\Gamma_{1,T+\tau}}N(x,y,\sigma-T-\tau)u^{q}(y,T+\tau)\,dS(y)\,d\tau \\
&\leq M(T)\int_{\O}N(x,y,\sigma-T)\,dy+M^{q}(\sigma)\int_{0}^{\sigma-T}\int_{\Gamma_{1,T+\tau}}N(x,y,\sigma-T-\tau)\,dS(y)\,d\tau.
\end{align*}
Applying part (e) in Corollary \ref{Cor, prop of NHK} yields $\int_{\O}N(x,y,\sigma-T)\,dy=1$, combining with the facts that $\Gamma_{1,T+\tau}\subseteq \Gamma_{1,T}$ and $N(x,y,\sigma-T-\tau)\geq 0$, we obtain
\[u(x,\sigma)\leq M(T)+M^{q}(\sigma)\int_{0}^{\sigma-T}\int_{\Gamma_{1,T}}N(x,y,\sigma-T-\tau)\,dS(y)\,d\tau.\]

Since $\sigma-T\leq t<1$, applying Lemma \ref{Lemma, bound by power of Gamma, NHK} with $\Gamma=\Gamma_{1,T}$ and $t=\sigma-T$ yields
\[u(x,\sigma) \leq M(T)+CM^{q}(\sigma)|\Gamma_{1,T}|^{\a}(\sigma-T)^{[1-(n-1)\a]/2}.\]
Then by using the fact $\sigma-T\leq t$ again, we conclude that
\[u(x,\sigma) \leq M(T)+CM^{q}(T+t)[A(T)]^{\a}t^{[1-(n-1)\a]/2}.\]
\end{itemize}

\end{proof}

Our strategy in the proof of Theorem \ref{Thm, prevent finite time blow-up} is to construct a strictly increasing sequence $\{M_{k}\}_{k\geq 0}$ such that the function $M(t)$ spends at least a certain time $t_{*}$ to increase from $M_{k-1}$ to $M_{k}$ when the surface area function $A(t)$ decreases at a certain speed. Denote $r_k=\frac{M_k}{M_{k-1}}$ to be the ratio in the k-th step. In order to secure a fixed lower bound $t_{*}$ in each step, there is a dilemma.
\begin{itemize}
\item [(1)] If $r_{k}$ is too small, then $A(t)$ is required to decay at a very fast speed in the current k-th step.

\item [(2)] If $r_{k}$ is too large, then $A(t)$ has to decay very fast in the future steps.
\end{itemize}
Thus, in order to avoid the super fast decay of $A(t)$ at any single step (since it may be difficult to achieve in practice), the growth rate $r_{k}$ has to be set delicately. We will discuss this issue in more details in Remark \ref{Re, sharpness of log growth} after the following proof.

\begin{proof}[Proof of Theorem \ref{Thm, prevent finite time blow-up}]
Denote $T_{0}=0$. For any $k\geq 1$, choose 
\be\label{def of M_(k), infty}
M_{k}=\ln[(k+1)e]M_{0}.\ee
Let $T_{k}$ be the first time that $M(t)$ reaches $M_{k}$ and define $t_{k}=T_{k}-T_{k-1}$. Since $u$ is continuous on $\ol{\O}\times[0,T^*)$,
\[T_{k}=\min\{t\geq 0: M(t)=M_k\}.\]
Fix $\beta>n-1$. In the proof below, $C$, $C_{1}$, $C_{2}$ and $C_3$ denote positive constants which only depend on $n$, $\O$, $q$ and $\beta$. 

Choose $\a $ to be any number between $\frac{1}{\beta}$ and $\frac{1}{n-1}$. Without loss of generality, we just fix $\a$ to be
\[\a =\frac{1}{2}\Big(\frac{1}{\beta}+\frac{1}{n-1}\Big).\]
For the convenience of applying (\ref{growth rate, general}) in Lemma \ref{Lemma, growth rate, general}, we define $\wt{\a}$ (corresponding to the power $\a$) as 
\[\wt{\a}=\frac{1-(n-1)\a }{2}=\frac{1}{4}\Big(1-\frac{n-1}{\beta}\Big).\]
Then the above powers $\a$ and $\wt{\a}$ satisfy
\be\label{range of para}
\frac{1}{\beta}<\a <\frac{1}{n-1}\quad\text{and}\quad 0<\wt{\a}<\frac{1}{4}.\ee
Hence, for any $k\geq 1$ such that $t_k\leq 1$, plugging $T=T_{k-1}$ and $t=t_{k}$ into Lemma \ref{Lemma, growth rate, general} leads to
\be\label{initial est, infty}
M_{k} \leq M_{k-1}+ C M_{k}^{q}[A(T_{k-1})]^{\a }t_{k}^{\wt{\a}}\ee
for some constant $C$.

We are trying to find constants $t_{*}\in(0,1]$ and $C^{*}>0$ depending only on $n$, $\O$, $q$, $\beta$, $|\Gamma_{1}|$ and $M_{0}$ such that if 
\be\label{decay of surface area, infty}
A(t)\leq |\Gamma_{1}|(1+C^{*}t)^{-\beta},\ee
then for any $k\geq 1$,
\be\label{induction, infty}
t_{k}\geq t_{*}.\ee
It is readily seen that as long as such $t_{*}$ and $C^{*}$ would be found, the theorem is justified. In the rest of the proof, the values of $t_{*}$ and $C^{*}$ will be determined via an induction argument. 

When $k=1$, if $t_{1}\geq 1$, then $t_{1}\geq t_{*}$ automatically holds since $t_{*}$ will be chosen in $(0,1]$, see (\ref{def of t_(*), infty}). If $t_1<1$, then  it follows from (\ref{initial est, infty}) that 
\[M_{1} \leq M_{0}+ C M_{1}^{q}|\Gamma_{1}|^{\a }t_{1}^{\wt{\a}}.\]
This implies
\[t_{1}\geq \bigg(\frac{M_{1}-M_{0}}{CM_{1}^{q}|\Gamma_{1}|^{\a }}\bigg)^{1/\wt{\a}}.\]
Due to the definition (\ref{def of M_(k), infty}),
\begin{align*}
t_{1} &\geq \bigg(\frac{(\ln 2)M_{0}}{C\ln^{q}(2e)M_{0}^{q}|\Gamma_{1}|^{\a }}\bigg)^{1/\wt{\a}}\\
&=C_{1}\big(M_{0}^{q-1}|\Gamma_{1}|^{\a }\big)^{-1/\wt{\a}},
\end{align*}
for some constant $C_{1}$. Denote 
\be\label{nota Y}
Y=M_{0}^{q-1}|\Gamma_{1}|^{\a }\ee
and define 
\be\label{def of t_(*), infty}
t_{*}=\min\Big\{1, C_{1}Y^{-1/\wt{\a}}\Big\}.\ee
Then (\ref{induction, infty}) holds for $k=1$.

Now suppose (\ref{induction, infty}) has been verified for $1\leq k\leq j$ with some $j\geq 1$, we are trying to prove (\ref{induction, infty}) for $k=j+1$. If $t_{j+1}\geq 1$, then again $t_{j+1}\geq t_{*}$ automatically holds. If $t_{j+1}<1$, then plugging $k=j+1$ into (\ref{initial est, infty}) yields
\be\label{recursive ineq, infty}
M_{j+1} \leq M_{j}+ C M_{j+1}^{q}[A(T_{j})]^{\a }t_{j+1}^{\wt{\a}}.\ee
Since $t_{k}\geq t_{*}$ for any $1\leq k\leq j$ by induction, then $T_{j}\geq jt_{*}$. Therefore, $A(T_{j})\leq A(jt_{*})$ due to the assumption (C). So (\ref{recursive ineq, infty}) leads to
\be\label{key step}
\frac{M_{j+1}-M_{j}}{M_{j+1}^{q}}\leq C[A(jt_{*})]^{\a }t_{j+1}^{\wt{\a}}.\ee
Taking advantage of the definition (\ref{def of M_(k), infty}) and the assumption (\ref{decay of surface area, infty}), we obtain 
\begin{align}\label{key step, simp}
\frac{\ln\big(\frac{j+2}{j+1}\big)}{\ln^{q}[(j+2)e] M_{0}^{q-1}}&\leq C|\Gamma_{1}|^{\a }(1+C^{*}jt_{*})^{-\beta\a }t_{j+1}^{\wt{\a}} \notag\\
&\leq C|\Gamma_{1}|^{\a }(C^{*}t_{*})^{-\beta\a }j^{-\beta\a }t_{j+1}^{\wt{\a}}.
\end{align}
Rearranging this inequality and recalling the notation $Y$ in (\ref{nota Y}),
\be\label{lower bound for t_(j+1)}
t_{j+1}^{\wt{\a}}\geq \frac{j^{\beta\a }\ln\big(\frac{j+2}{j+1}\big)}{\ln^{q}[(j+2)e]}\,\frac{(C^{*}t_{*})^{\beta\a }}{CY}.\ee
Since (\ref{range of para}) implies $\beta\a >1$, then 
\be\label{sequence goes to infty}
\lim_{j\rightarrow\infty}\frac{j^{\beta\a }\ln\big(\frac{j+2}{j+1}\big)}{\ln^{q}[(j+2)e]}=\infty.\ee
As a result, there exists a uniform positive lower bound (depending on $n$, $\beta$ and $q$) for 
\[\mbox{$j^{\beta\a }\ln\big(\frac{j+2}{j+1}\big)\big/\ln^{q}[(j+2)e]$}\]
when $j\geq 1$. So it follows from (\ref{lower bound for t_(j+1)}) that 
\[t_{j+1}^{\wt{\a}}\geq \frac{C_{2}(C^{*}t_{*})^{\beta\a }}{Y}.\]
for some constant $C_{2}$. In order for $t_{j+1}\geq t_{*}$, if suffices to have 
\[\frac{C_{2}(C^{*}t_{*})^{\beta\a }}{Y}\geq t_{*}^{\wt{\a}}.\]
Equivalently,
\be\label{est for C^(*)}
C^{*}\geq \bigg(\frac{Yt_{*}^{\wt{\a}-\beta\a }}{C_{2}}\bigg)^{\frac{1}{\beta\a }}.\ee
Noticing $\wt{\a}<\frac14<\beta\a $ and recalling the choice (\ref{def of t_(*), infty}) for $t_{*}$, then (\ref{est for C^(*)}) becomes
\begin{align}
C^{*} &\geq \bigg(\frac{Y}{C_{2}}\bigg)^{\frac{1}{\beta\a }}\max\Big\{1,\,C_{1}^{\frac{\wt{\a}}{\beta\a }-1}Y^{\frac{1}{\wt{\a}}-\frac{1}{\beta\a }}\Big\} \nonumber\\
&=C_{2}^{-\frac{1}{\beta\a }}\max\Big\{Y^{\frac{1}{\beta\a }}, \, C_{1}^{\frac{\wt{\a}}{\beta\a }-1}Y^{\frac{1}{\wt{\a}}}\Big\}. \label{cond on C^(*)}
\end{align}
Hence, the theorem is justified by defining 
\be\label{def of C^*}
C^{*}=C_3\max\big\{Y^{\frac{1}{\beta\a }}, \, Y^{\frac{1}{\wt{\a}}}\big\},\ee
where $C_{3}:=C_{2}^{-\frac{1}{\beta\a }}\max\big\{1,\,C_{1}^{\frac{\wt{\a}}{\beta\a }-1}\big\}$.
\end{proof}

\begin{remark}\label{Re, sharpness of log growth}
From the above proof, we can see from (\ref{key step}) that the crucial ingredient is to find lower bounds for $\Lambda_{j}$, where 
\[\Lambda_{j}:=\frac{M_{j}-M_{j-1}}{M_{j}^{q}},\quad\forall \,j\geq 1.\]
On the one hand, it follows from Lemma \ref{Lemma, sharp grow rate, seq} (see Section \ref{Sec, fun lemma}) that for any $\eps>0$, 
\be\label{small seq}
\Lambda_{j}<\frac{\eps}{j} \quad \text{occurs infinitely many often.} \ee  
On the other hand, if choosing $M_{j}\sim \ln(j+2)$, then it is readily seen that for any $\delta>0$, 
\be\label{not so small}
\lim_{j\to\infty} j^{1+\delta}\Lambda_{j}=\infty.\ee
Combining (\ref{small seq}) and (\ref{not so small}) together, the choice of a sequence $\{M_j\}$ with logarithmic growth seems optimal. This is why we define the sequence $\{M_j\}$ as in (\ref{def of M_(k), infty}).
\end{remark}

\begin{remark}\label{Re, choice of beta}
As a consequence of the above argument, with the choice (\ref{def of M_(k), infty}) for $\{M_{j}\}$, the corresponding $\Lambda_{j}$ decays almost at the rate of $\frac1j$. Meanwhile, due to the decay (\ref{decay of surface area, infty}) of the surface area $A(t)$, we can see from (\ref{key step, simp}) that the power gain on $j$ is $\b\a$. Hence, in order to compensate the $\frac{1}{j}$ decay of $\Lambda_{j}$, $\b\a$ has to be greater than 1. On the other hand, the index $\a$ can not be made larger than $\frac{1}{n-1}$ due to Lemma \ref{Lemma, bound by power of Gamma, NHK} (also see Proposition 5.1 in \cite{YZ19} for a discussion on the sharpness of this upper bound $\frac{1}{n-1}$). Therefore, $\b$ has to be chosen greater than $n-1$.
\end{remark}

\section{Proof of Theorem \ref{Thm, control temp}}
\label{Sec, proof of bdd soln}
The essential idea is similar to that in the proof of Theorem \ref{Thm, prevent finite time blow-up}, but the choice of the sequence $\{M_{k}\}_{k\geq 1}$ will be much more complicated. In the proof of Theorem \ref{Thm, prevent finite time blow-up}, $\{M_{k}\}_{k\geq 1}$ is still allowed to increase logarithmically. However, in the proof of Theorem \ref{Thm, control temp}, the growth rate has to be slower since $\{M_{k}\}_{k\geq 1}$ is bounded by $B$. As a result, delicate adjustment is needed in the definition of $\{M_{k}\}_{k\geq 1}$, see (\ref{def of M_k, control temp}).

\begin{proof}[Proof of Theorem \ref{Thm, control temp}]
Choose $s$ to be any number such that 
\be\label{range of s}
0<s<\frac{\b}{n-1}-1.\ee
Without loss of generality, we just fix $s$ to be 
\be\label{def of s}
s=\frac{1}{2}\Big(\frac{\b}{n-1}-1\Big).\ee
So this number $s$ only depends on $n$ and $\b$.
Next, we define a function $g_{s}:(0,\infty)\rightarrow (1,\infty)$ by 
\be\label{def of aux fcn}
g_{s}(\lam)=\sum_{m=0}^{\infty}\frac{1}{(1+m)(1+\lam m)^s}.\ee
It is readily seen that $g_{s}$ is a decreasing and continuous bijection from $(0,\infty)$ to $(1,\infty)$. So it is valid to define 
\be\label{lam_B}
\lam_{B}=g_{s}^{-1}\Big(\frac{B}{M_0}\Big),\ee
which only depends on $n$, $\b$ and $\frac{B}{M_0}$. In addition,
\be\label{end behavior of lam_B}
\lim_{B\rightarrow M_0^+}\lam_{B}=\infty\quad\text{and}\quad \lim_{B\rightarrow\infty}\lam_{B}=0.\ee
After these preparation, we define $M_{k}$ by
\be\label{def of M_k, control temp}
M_k=M_0\sum_{m=0}^{k}\frac{1}{(1+m)(1+\lam_{B}m)^s}, \quad\forall\, k\geq 0.\ee
Similar to the proof of Theorem \ref{Thm, prevent finite time blow-up}, we define $T_k$ to be the first time that $M(t)$ reaches $M_{k}$ and denote $t_k=T_k-T_{k-1}$. In the proof below, $C$ will denote a positive constant which only depends on $n$, $\O$, $q$ and $\b$. $C_1$, $C_2$ and $C_3$ will denote positive constants which additionally depend on $\frac{B}{M_0}$.

We choose $\a$ such that 
\be\label{range for alpha}
\frac{1+s}{\b}<\a<\frac{1}{n-1}.\ee
Without loss of generality, we just fix $\a$ to be 
\be\label{def of alpha}
\a=\frac{1}{2}\Big(\frac{1+s}{\b}+\frac{1}{n-1}\Big)=\frac{1}{4}\Big(\frac{1}{\b}+\frac{3}{n-1}\Big).\ee
Then the corresponding power $\wt{\a}$ to $\a$ is defined as 
\[\wt{\a}=\frac{1-(n-1)\a}{2}=\frac{1}{8}\Big(1-\frac{n-1}{\b}\Big).\]
Thus, for any $k\geq 1$ such that $t_k\leq 1$, plugging $T=T_{k-1}$ and $t=t_{k}$ into Lemma \ref{Lemma, growth rate, general} leads to
\be\label{initial est, control}
M_{k} \leq M_{k-1}+ C M_{k}^{q}[A(T_{k-1})]^{\a }t_{k}^{\wt{\a}}\ee
for some constant $C$. 

We are trying to find constants $t_{*}\in(0,1]$ and $C_{B}^{*}>0$, depending only on $n$, $\O$, $q$, $\beta$, $|\Gamma_{1}|$, $M_{0}$ and $B$,  such that if 
\be\label{decay of surface area, control}
A(t)\leq |\Gamma_{1}|(1+C_{B}^{*}t)^{-\beta},\ee
then for any $k\geq 1$,
\be\label{induction, control}
t_{k}\geq t_{*}.\ee
It is readily seen that as long as such $t_{*}$ and $C_{B}^{*}$ would be found, the theorem is justified. In the rest of the proof, the values of $t_{*}$ and $C_{B}^{*}$ will be determined via an induction argument. 

When $k=1$, if $t_{1}\geq 1$, then $t_{1}\geq t_{*}$ automatically holds since $t_{*}$ will be chosen in $(0,1]$, see (\ref{def of t_(*), control}). If $t_1<1$, then  it follows from (\ref{initial est, control}) that 
\[M_{1} \leq M_{0}+ C M_{1}^{q}|\Gamma_{1}|^{\a }t_{1}^{\wt{\a}}.\]
This yields
\[t_{1}\geq \bigg(\frac{M_{1}-M_{0}}{CM_{1}^{q}|\Gamma_{1}|^{\a }}\bigg)^{1/\wt{\a}}.\]
Recalling the definition (\ref{def of M_k, control temp}) for $M_1$, 
\[t_{1} \geq \bigg(\frac{2^{q-1}}{CM_{0}^{q-1}|\Gamma_{1}|^{\a }}\frac{(1+\lam_B)^{(q-1)s}}{[1+2(1+\lam_B)^s]^q}\bigg)^{1/\wt{\a}}.\]
Writing
\be\label{nota Y, control}
Y=M_{0}^{q-1}|\Gamma_{1}|^{\a }\ee
and denoting
\be\label{def of C_1}
C_{1}=\bigg(\frac{2^{q-1}}{C}\frac{(1+\lam_B)^{(q-1)s}}{[1+2(1+\lam_B)^s]^q}\bigg)^{1/\wt{\a}},\ee
then 
\[t_{1}\geq C_{1}Y^{-1/\wt{\a}}.\]
Define 
\be\label{def of t_(*), control}
t_{*}=\min\Big\{1, C_{1}Y^{-1/\wt{\a}}\Big\}.\ee
With this choice of $t_{*}$, the induction (\ref{induction, control}) holds for $k=1$.

Now suppose (\ref{induction, control}) has been verified for $1\leq k\leq j$ with some $j\geq 1$, we are trying to prove (\ref{induction, control}) for $k=j+1$. If $t_{j+1}\geq 1$, then again $t_{j+1}\geq t_{*}$ automatically holds. If $t_{j+1}<1$, then plugging $k=j+1$ into (\ref{initial est, control}) yields
\be\label{recursive ineq, control}
M_{j+1} \leq M_{j}+ C M_{j+1}^{q}[A(T_{j})]^{\a }t_{j+1}^{\wt{\a}}.\ee
Since $t_{k}\geq t_{*}$ for any $1\leq k\leq j$ by induction, then $T_{j}\geq jt_{*}$. Therefore $A(T_{j})\leq A(jt_{*})$ due to the assumption (C). So (\ref{recursive ineq, control}) leads to
\be\label{est in j step}
\frac{M_{j+1}-M_{j}}{M_{j+1}^{q}}\leq C[A(jt_{*})]^{\a }t_{j+1}^{\wt{\a}}.\ee
Recalling the definition (\ref{def of M_k, control temp}) and using the fact $1+\lam_{B}(j+1)\leq (1+\lam_{B})(j+2)$, we have 
\[M_{j+1}-M_{j}=\frac{M_0}{(j+2)[1+\lam_B(j+1)]^s}\geq \frac{M_0}{(j+2)^{1+s}(1+\lam_B)^s}\]
and 
\[M_{j+1}\leq M_0\sum_{m=0}^{j+1}\frac{1}{1+m}\leq M_0\ln[(j+2)e]. \]
Therefore, 
\be\label{lower bdd for ratio j step}\frac{M_{j+1}-M_j}{M_{j+1}^{q}}\geq \frac{1}{M_0^{q-1}\ln^{q}[(j+2)e]\,(j+2)^{1+s}(1+\lam_B)^s}.\ee
Now taking advantage of the assumption (\ref{decay of surface area, control}), we obtain
\be\label{area bdd j step}
[A(jt_*)]^{-\a} \geq |\Gamma_1|^{-\a}(1+C_{B}^{*}jt_*)^{\b\a}\geq |\Gamma_1|^{-\a}(C_B^*)^{\b\a}j^{\b\a}t_*^{\b\a}.\ee
Combining (\ref{est in j step}), (\ref{lower bdd for ratio j step}) and (\ref{area bdd j step}) together yields 
\[t_{j+1}^{\wt{\a}} \geq \frac{(C_B^*)^{\b\a}j^{\b\a}t_*^{\b\a}}{C|\Gamma_1|^{\a}M_0^{q-1}\ln^{q}[(j+2)e]\,(j+2)^{1+s}(1+\lam_B)^s}.\]
Recalling $Y=|\Gamma_{1}|^{\a}M_0^{q-1}$, then rearranging the right hand side of the above inequality leads to
\be\label{lower bdd j step, 1}
t_{j+1}^{\wt{\a}} \geq \frac{j^{\b\a}}{\ln^{q}[(j+2)e]\,(j+2)^{1+s}}\frac{(C_B^*)^{\b\a}t_*^{\b\a}}{C(1+\lam_B)^s Y}.\ee
Since (\ref{range for alpha}) implies $\beta\a >1+s$, then 
\be\label{sequence goes to infty, control}
\lim_{j\rightarrow\infty}\frac{j^{\b\a}}{\ln^{q}[(j+2)e]\,(j+2)^{1+s}}=\infty.\ee
Hence, there exists a uniform positive lower bound (depending on $n$, $\b$ and $q$) for
\[\dfrac{j^{\b\a}}{\ln^{q}[(j+2)e]\,(j+2)^{1+s}}\]
when $j\geq 1$. So it follows from (\ref{lower bdd j step, 1}) that 
\[t_{j+1}^{\wt{\a}}\geq \frac{C(C_B^*)^{\b\a}t_*^{\b\a}}{(1+\lam_B)^s Y}\]
for another constant $C$ which only depends on $n$, $\O$, $q$ and $\b$. Writing 
\be\label{def of C_2}
C_{2}=\frac{C}{(1+\lam_B)^s}\ee
to be a constant which additionally depends on $\frac{B}{M_0}$, then 
\[t_{j+1}^{\wt{\a}}\geq \frac{C_2(C_B^*)^{\b\a}t_*^{\b\a}}{Y}.\]
In order for $t_{j+1}\geq t_{*}$, if suffices to have 
\[\frac{C_2(C_B^*)^{\b\a}t_*^{\b\a}}{Y}\geq t_{*}^{\wt{\a}}.\]
Equivalently,
\be\label{est for C_B^*}
C_B^{*}\geq \Big(\frac{Y}{C_2}\Big)^{\frac{1}{\b\a}}t_{*}^{\frac{\wt{\a}}{\b\a}-1}.\ee
Noticing $\wt{\a}<\frac18<\beta\a $ and recalling the choice (\ref{def of t_(*), control}) for $t_{*}$, then (\ref{est for C_B^*}) becomes
\begin{align}
C_B^{*} &\geq \bigg(\frac{Y}{C_{2}}\bigg)^{\frac{1}{\beta\a }}\max\Big\{1,\,C_{1}^{\frac{\wt{\a}}{\beta\a }-1}Y^{\frac{1}{\wt{\a}}-\frac{1}{\beta\a }}\Big\} \nonumber\\
&=C_{2}^{-\frac{1}{\beta\a }}\max\Big\{Y^{\frac{1}{\beta\a }}, \, C_{1}^{\frac{\wt{\a}}{\beta\a }-1}Y^{\frac{1}{\wt{\a}}}\Big\}. \label{cond on C_B^*}
\end{align}
Therefore, the theorem is justified by defining 
\be\label{def of C_B^*}
C_B^{*}=C_3\max\big\{Y^{\frac{1}{\beta\a }}, \, Y^{\frac{1}{\wt{\a}}}\big\},\ee
where 
\be\label{def of C_3}
C_{3}:=C_{2}^{-\frac{1}{\beta\a }}\max\big\{1,\,C_{1}^{\frac{\wt{\a}}{\beta\a }-1}\big\}.\ee
\end{proof}

\begin{remark}
Fix $n$, $\O$, $q$, $\b$, $M_0$ and $|\Gamma_1|$. Let  $B$ vary in the range $(M_0,\infty)$. Based on the definitions (\ref{def of C_1}), (\ref{def of C_2}) and (\ref{def of C_3}) for the constants $C_{1}$, $C_{2}$ and $C_{3}$ in the above proof, it follows from (\ref{end behavior of lam_B}) that 
\be\label{end behavior or const}\begin{split}
\lim_{B\rightarrow M_0^+}C_1=0, &\qquad \lim_{B\rightarrow\infty} C_1=\text{a positive constant},\\
\lim_{B\rightarrow M_0^+}C_2=0, &\qquad \lim_{B\rightarrow\infty} C_2=\text{a positive constant},\\
\lim_{B\rightarrow M_0^+}C_3=\infty, &\qquad \lim_{B\rightarrow\infty} C_3=\text{a positive constant}.
\end{split}\ee
As a result, it follow from (\ref{def of C_B^*}) and (\ref{end behavior or const}) that 
\be\label{end behavior of C_B^*}
\lim_{B\rightarrow M_0^+}C_B^*=\infty, \qquad \lim_{B\rightarrow\infty} C_B^*=\text{a positive constant}.\ee
\begin{itemize}
\item For the first relation $\lim\limits_{B\rightarrow M_0^+}C_B^*=\infty$ in (\ref{end behavior of C_B^*}), it accords with our expectation since the surface area has to decay super fast if the temperature is barely allowed to increase.

\item For the second relation $\lim\limits_{B\rightarrow\infty} C_B^*=\text{a positive constant}$ in (\ref{end behavior of C_B^*}), it matches the conclusion in Theorem \ref{Thm, prevent finite time blow-up}.
\end{itemize}
\end{remark}

\section{An elementary lemma}
\label{Sec, fun lemma}
This section will provide an auxiliary lemma to support the arguments in Remark \ref{Re, sharpness of log growth} and \ref{Re, choice of beta} which explained why the power $\b$ in Theorem \ref{Thm, prevent finite time blow-up} has to be greater than $n-1$ if using the method in this paper.

For any positive and increasing sequence $\{M_{j}\}_{j\geq 1}$, the growth rate at the $j$th term is usually defined as $\frac{M_j-M_{j-1}}{M_{j-1}}$. Given any $q>1$, one of the key steps in the proofs of Theorem \ref{Thm, prevent finite time blow-up} and \ref{Thm, control temp} is to estimate how large the $\Lambda_{j}$ is, where 
\be\label{def of Lam}
\Lambda_j:=\frac{M_j-M_{j-1}}{M_{j}^q}.\ee
When $M_j$ is close to $M_{j-1}$, $\Lambda_j$ can be regarded as a nonlinear analogue of the growth rate at the $j$th term. But when $M_j$ is much larger than $M_{j-1}$, $\Lambda_j$ should be barely called a growth rate.

On the other hand, it is readily seen from (\ref{def of Lam}) that 
\[\Lambda_{j}\leq \min\bigg\{\frac{M_j-M_{j-1}}{M_{1}^q},\, \frac{1}{M_{j}^{q-1}}\bigg\}.\]
Hence,
$\Lambda_j\to 0^{+}$ as $j\to \infty$, no matter $\{M_j\}_{j\geq 1}$ is a bounded sequence or not. But how fast $\Lambda_j$ converges to 0? The following lemma concludes that for any $\eps>0$, $\Lambda_j<\frac{\eps}{j}$ for infinitely many $j$'s.

\begin{lemma}\label{Lemma, sharp grow rate, seq}
For any $q>1$ and for any positive and increasing sequence $\{M_{j}\}_{j\geq 1}$,
\be\label{infimum is zero}
\liminf_{j\to \infty}\,\frac{j(M_{j}-M_{j-1})}{M_{j}^{q}}=0.\ee
\end{lemma}
\begin{proof}
If (\ref{infimum is zero}) does not hold, then there exists $\eps>0$ and $N>0$ such that for any $j\geq N$,
\be\label{lower bdd for seq}
\frac{j(M_{j}-M_{j-1})}{M_{j}^{q}}\geq \eps.\ee
Fix the above $\eps$ and $N$ in the rest of the proof.

We first claim that (\ref{lower bdd for seq}) implies that 
\be\label{sup is inf}
\sup_{j\geq 1}M_j=\infty.\ee
In fact, since $\{M_j\}_{j\geq 1}$ is an increasing sequence, it follows from (\ref{lower bdd for seq}) that for any $j\geq 2$, 
\be\label{diff large}
M_j-M_{j-1}\geq \frac{\eps M_j^q}{j}\geq \frac{\eps M_1^q}{j}.\ee
As a result, for any $j\geq 2$, 
\begin{align*}
M_j =M_1+\sum_{k=2}^{j}(M_k-M_{k-1}) &\geq M_{1}+\sum_{k=2}^{j}\frac{\eps M_1^q}{k}= M_1+\eps M_1^q\sum_{k=2}^{j}\frac1k,
\end{align*}
which implies (\ref{sup is inf}).

Due to (\ref{sup is inf}), there exists $N_1>N$ such that $M_j\geq 1$ for any $j\geq N_1$. Define 
\[\wt{q}=\min\{q, 2\}.\]
The reason for introducing this $\wt{q}$ is to apply the mean value theorem in (\ref{mvt}). For any $j\geq N_1$, it follows from (\ref{lower bdd for seq}) that 
\[j(M_{j}-M_{j-1})\geq \eps M_{j}^{q}\geq \eps M_{j}^{\wt{q}}.\]
Rearranging the above inequality leads to
\[M_{j-1}\leq M_{j}\bigg(1-\frac{\eps M_{j}^{\wt{q}-1}}{j}\bigg).\]
Raising both sides to the power $\wt{q}-1$ and multiplying by $\eps$ yields
\[\eps M_{j-1}^{\wt{q}-1}\leq \eps M_{j}^{\wt{q}-1}\bigg(1-\frac{\eps M_{j}^{\wt{q}-1}}{j}\bigg)^{\wt{q}-1}.\]
Define \[x_{j}=\eps M_{j}^{\wt{q}-1}.\]
Then for any $j\geq N_1$,
\be\label{ineq for x_j}
x_{j-1}\leq x_{j}\Big(1-\frac{x_{j}}{j}\Big)^{\wt{q}-1}. \ee
Since $1<\wt{q}\leq 2$, then by the mean value theorem,
\be\label{mvt}
1-\Big(1-\frac{x_j}{j}\Big)^{\wt{q}-1}\geq (\wt{q}-1)\frac{x_j}{j}.\ee
Hence, it follows from (\ref{ineq for x_j}) and (\ref{mvt}) that
\[x_{j-1}\leq x_{j}\Big[1-(\wt{q}-1)\,\frac{x_{j}}{j}\Big]=\frac{x_{j}[j-(\wt{q}-1)x_{j}]}{j}.\]
Taking reciprocal,
\begin{align}
\frac{1}{x_{j-1}} &\geq \frac{j}{x_{j}[j-(\wt{q}-1)x_{j}]} \nonumber\\
&=\frac{1}{x_{j}}+\frac{\wt{q}-1}{j-(\wt{q}-1)x_{j}}\nonumber\\
&\geq \frac{1}{x_{j}}+\frac{\wt{q}-1}{j}. \label{recuisive ineq}
\end{align}
Now for any $k>N_1$, we apply (\ref{recuisive ineq}) repeatedly from $j=N_1$ to $j=k$ and add all these inequalities together, then
\[\frac{1}{x_{N_1-1}}\geq \frac{1}{x_{k}}+(\wt{q}-1)\sum_{j=N_1}^{k}\frac{1}{j}.\]
Sending $k\rightarrow \infty$, the left hand side of the above inequality is a constant while the right hand side tends to infinity, which is a contradiction. Thus, (\ref{infimum is zero}) is verified.
\end{proof}
 
\begin{remark}
If we choose $M_{j}=\ln(j+1)$ for any $j\geq 1$, then for any $\eps>0$, 
\be\label{infimum is infty}
\liminf_{j\to \infty}\,\frac{j^{1+\eps}(M_{j}-M_{j-1})}{M_{j}^{q}}=\infty.\ee
Therefore, the result in Lemma \ref{Lemma, sharp grow rate, seq} is sharp in the sense that the term $j$ on the numerator in (\ref{infimum is zero}) can not be improved to any higher power $j^{1+\eps}$ with $\eps>0$.
\end{remark}

\bigskip\bigskip

{\small
\bibliographystyle{plain}
\bibliography{Ref-Parabolic}

\begin{thebibliography}{10}

\bibitem{AD17}
J.~R. Anderson and K.~Deng.
\newblock A lower bound on the blow up time for solutions of a chemotaxis
  system with nonlinear chemotactic sensitivity.
\newblock {\em Nonlinear Anal.}, 159:2--9, 2017.

\bibitem{BS14}
A.~Bao and X.~Song.
\newblock Bounds for the blowup time of the solutions to quasi-linear parabolic
  problems.
\newblock {\em Z. Angew. Math. Phys.}, 65(1):115--123, 2014.

\bibitem{DL00}
K.~Deng and H.~A. Levine.
\newblock The role of critical exponents in blow-up theorems: the sequel.
\newblock {\em J. Math. Anal. Appl.}, 243(1):85--126, 2000.

\bibitem{DS16}
J.~Ding and X.~Shen.
\newblock Blow-up in {$p$}-{L}aplacian heat equations with nonlinear boundary
  conditions.
\newblock {\em Z. Angew. Math. Phys.}, 67(5):Art. 125, 18, 2016.

\bibitem{Dre40}
F.~G. Dressel.
\newblock The fundamental solution of the parabolic equation.
\newblock {\em Duke Math. J.}, 7:186--203, 1940.

\bibitem{Ena11}
C.~Enache.
\newblock Blow-up phenomena for a class of quasilinear parabolic problems under
  {R}obin boundary condition.
\newblock {\em Appl. Math. Lett.}, 24(3):288--292, 2011.

\bibitem{Fel37}
W.~Feller.
\newblock Zur {T}heorie der stochastischen {P}rozesse.
\newblock {\em Math. Ann.}, 113(1):113--160, 1937.

\bibitem{Fri64}
A.~Friedman.
\newblock {\em Partial differential equations of parabolic type}.
\newblock Prentice-Hall, Inc., Englewood Cliffs, N.J., 1964.

\bibitem{Fuj66}
H.~Fujita.
\newblock On the blowing up of solutions of the {C}auchy problem for
  {$u_{t}=\Delta u+u^{1+\alpha }$}.
\newblock {\em J. Fac. Sci. Univ. Tokyo Sect. I}, 13:109--124, 1966.

\bibitem{Hu11}
B.~Hu.
\newblock {\em Blow-up theories for semilinear parabolic equations}, volume
  2018 of {\em Lecture Notes in Mathematics}.
\newblock Springer, Heidelberg, 2011.

\bibitem{HY94}
B.~Hu and H.-M. Yin.
\newblock The profile near blowup time for solution of the heat equation with a
  nonlinear boundary condition.
\newblock {\em Trans. Amer. Math. Soc.}, 346(1):117--135, 1994.

\bibitem{Ito53}
S.~Ito.
\newblock The fundamental solution of the parabolic equation in a
  differentiable manifold.
\newblock {\em Osaka Math. J.}, 5:75--92, 1953.

\bibitem{Ito54}
S.~Ito.
\newblock The fundamental solution of the parabolic equation in a
  differentiable manifold. {II}.
\newblock {\em Osaka Math. J.}, 6:167--185, 1954.

\bibitem{Ito57a}
S.~Ito.
\newblock A boundary value problem of partial differential equations of
  parabolic type.
\newblock {\em Duke Math. J.}, 24(3):299--312, 1957.

\bibitem{Ito57b}
S.~Ito.
\newblock Fundamental solutions of parabolic differential equations and
  boundary value problems.
\newblock {\em Jap. J. Math.}, 27:55--102, 1957.

\bibitem{Kap63}
S.~Kaplan.
\newblock On the growth of solutions of quasi-linear parabolic equations.
\newblock {\em Comm. Pure Appl. Math.}, 16:305--330, 1963.

\bibitem{LSU68}
O.~A. Lady\v{z}enskaja, V.~A. Solonnikov, and N.~N. Ural'ceva.
\newblock {\em Linear and quasilinear equations of parabolic type}.
\newblock Translated from the Russian by S. Smith. Translations of Mathematical
  Monographs, Vol. 23. American Mathematical Society, Providence, R.I., 1968.

\bibitem{Lev75}
H.~A. Levine.
\newblock Nonexistence of global weak solutions to some properly and improperly
  posed problems of mathematical physics: the method of unbounded {F}ourier
  coefficients.
\newblock {\em Math. Ann.}, 214(3):205--220, 1975.

\bibitem{Lev90}
H.~A. Levine.
\newblock The role of critical exponents in blowup theorems.
\newblock {\em SIAM Rev.}, 32(2):262--288, 1990.

\bibitem{LP74}
H.~A. Levine and L.~E. Payne.
\newblock Nonexistence theorems for the heat equation with nonlinear boundary
  conditions and for the porous medium equation backward in time.
\newblock {\em J. Differential Equations}, 16:319--334, 1974.

\bibitem{LL12}
F.~Li and J.~Li.
\newblock Global existence and blow-up phenomena for nonlinear divergence form
  parabolic equations with inhomogeneous {N}eumann boundary conditions.
\newblock {\em J. Math. Anal. Appl.}, 385(2):1005--1014, 2012.

\bibitem{Lie96}
G.~M. Lieberman.
\newblock {\em Second order parabolic differential equations}.
\newblock World Scientific Publishing Co., Inc., River Edge, NJ, 1996.

\bibitem{L-GMW91}
J.~L{\'o}pez-G{\'o}mez, V.~M{\'a}rquez, and N.~Wolanski.
\newblock Blow up results and localization of blow up points for the heat
  equation with a nonlinear boundary condition.
\newblock {\em J. Differential Equations}, 92(2):384--401, 1991.

\bibitem{NY19}
T.~Nishino and T.~Yokota.
\newblock Effect of nonlinear diffusion on a lower bound for the blow-up time
  in a fully parabolic chemotaxis system.
\newblock {\em J. Math. Anal. Appl.}, 479(1):1078--1098, 2019.

\bibitem{PP13}
L.~E. Payne and G.~A. Philippin.
\newblock Blow-up phenomena in parabolic problems with time dependent
  coefficients under {D}irichlet boundary conditions.
\newblock {\em Proc. Amer. Math. Soc.}, 141(7):2309--2318, 2013.

\bibitem{PPV-P10a}
L.~E. Payne, G.~A. Philippin, and S.~Vernier~Piro.
\newblock Blow-up phenomena for a semilinear heat equation with nonlinear
  boundary condition, {I}.
\newblock {\em Z. Angew. Math. Phys.}, 61(6):999--1007, 2010.

\bibitem{PPV-P10b}
L.~E. Payne, G.~A. Philippin, and S.~Vernier~Piro.
\newblock Blow-up phenomena for a semilinear heat equation with nonlinear
  boundary condition, {II}.
\newblock {\em Nonlinear Anal.}, 73(4):971--978, 2010.

\bibitem{PS06}
L.~E. Payne and P.~W. Schaefer.
\newblock Lower bounds for blow-up time in parabolic problems under {N}eumann
  conditions.
\newblock {\em Appl. Anal.}, 85(10):1301--1311, 2006.

\bibitem{PS07}
L.~E. Payne and P.~W. Schaefer.
\newblock Lower bounds for blow-up time in parabolic problems under {D}irichlet
  conditions.
\newblock {\em J. Math. Anal. Appl.}, 328(2):1196--1205, 2007.

\bibitem{PS09}
L.~E. Payne and P.~W. Schaefer.
\newblock Bounds for blow-up time for the heat equation under nonlinear
  boundary conditions.
\newblock {\em Proc. Roy. Soc. Edinburgh Sect. A}, 139(6):1289--1296, 2009.

\bibitem{QS07}
P.~Quittner and P.~Souplet.
\newblock {\em Superlinear parabolic problems. Blow-up, global existence and
  steady states.}
\newblock Birkh\"auser Advanced Texts: Basler Lehrb\"ucher. [Birkh\"auser
  Advanced Texts: Basel Textbooks]. Birkh\"auser Verlag, Basel, 2007.

\bibitem{RR97}
D.~F. Rial and J.~D. Rossi.
\newblock Blow-up results and localization of blow-up points in an
  {$N$}-dimensional smooth domain.
\newblock {\em Duke Math. J.}, 88(2):391--405, 1997.

\bibitem{TV-P16}
Y.~Tao and S.~Vernier~Piro.
\newblock Explicit lower bound of blow-up time in a fully parabolic chemotaxis
  system with nonlinear cross-diffusion.
\newblock {\em J. Math. Anal. Appl.}, 436(1):16--28, 2016.

\bibitem{Wal75}
W.~Walter.
\newblock On existence and nonexistence in the large of solutions of parabolic
  differential equations with a nonlinear boundary condition.
\newblock {\em SIAM J. Math. Anal.}, 6:85--90, 1975.

\bibitem{YZ16}
X.~Yang and Z.~Zhou.
\newblock Blow-up problems for the heat equation with a local nonlinear
  {N}eumann boundary condition.
\newblock {\em J. Differential Equations}, 261(5):2738--2783, 2016.

\bibitem{YZ18}
X.~Yang and Z.~Zhou.
\newblock Improvements on lower bounds for the blow-up time under local
  nonlinear {N}eumann conditions.
\newblock {\em J. Differential Equations}, 265(3):830--862, 2018.

\bibitem{YZ19}
X.~Yang and Z.~Zhou.
\newblock Lifespan estimates via {N}eumann heat kernel.
\newblock {\em Z. Angew. Math. Phys.}, 70(1):Art. 30, 2019.

\end{thebibliography}
}

\bigskip

\thanks{(X. Yang) Department of Mathematics, Virginia Polytechnic Institute and State University, Blacksburg, VA 24061, USA.}

\thanks{Email: xinyangmath@vt.edu}

\bigskip

\thanks{(Z. Zhou) Department of Mathematics, Michigan State University, East Lansing, MI 48824, USA.} 

\thanks{Email: zfzhou@math.msu.edu}
\end{document}